\documentclass[a4paper,12pt]{amsart}

\usepackage{amsfonts,latexsym,rawfonts,amsmath,amssymb,amsthm, mathrsfs, lscape}
\usepackage[all]{xy}
\usepackage[english]{babel}
\usepackage[utf8]{inputenc}

\usepackage{array, tabularx}

\usepackage{setspace}
\usepackage{textcomp}

\usepackage[hypertexnames=false,
backref=page,
    pdftex,
    pdfpagemode=UseNone,	
    breaklinks=true,
    extension=pdf,
    colorlinks=true,
    linkcolor=blue,
    citecolor=blue,
    urlcolor=blue,
]{hyperref}

\renewcommand{\Re}{\mathsf{Re}\,}

\usepackage{color}

\usepackage[top=1.5in, bottom=1in, left=0.9in, right=0.9in]{geometry}

\newcommand{\referenza}{}

\newtheorem{thm}{Theorem}
\newtheorem*{thm*}{Theorem \referenza}
\newtheorem{thmm}{Theorem}
\newtheorem{cor}[thm]{Corollary}
\newtheorem*{cor*}{Corollary \referenza}

\newtheorem*{lem*}{Lemma \referenza}

\newtheorem*{prop*}{Proposition \referenza}

\newtheorem*{conj*}{Conjecture \referenza}
\newtheorem{rmk}[thm]{Remark}

\newtheorem{defi}[thm]{Definition}

\numberwithin{equation}{section}

\def \N {\mathbb N}

\def \R {\mathbb R}
\def \C {\mathbb C}
\def \Z {\mathbb Z}

\def\XXint#1#2#3{{\setbox0=\hbox{$#1{#2#3}{\int}$ }
\vcenter{\hbox{$#2#3$ }}\kern-.6\wd0}}

\allowdisplaybreaks

\title{Remarks on Chern-Einstein Hermitian metrics}

\author{Daniele Angella}

\address[Daniele Angella]{Dipartimento di Matematica e Informatica "Ulisse Dini"\\
Università di Firenze\\
viale Morgagni 67/a\\
50134 Firenze, Italy
}
\email{daniele.angella@gmail.com}
\email{daniele.angella@unifi.it}

\author{Simone Calamai}
\address[Simone Calamai]{Dipartimento di Matematica e Informatica "Ulisse Dini"\\
Università di Firenze\\
via Morgagni 67/A\\
50134 Firenze, Italy}
\email{scalamai@math.unifi.it}
\email{simocala@gmail.com}

\author{Cristiano Spotti}
\address[Cristiano Spotti]{Department of Mathematics - Centre for Quantum Geometry of Moduli Spaces\\
Aarhus Universitet\\
Building 1530, Office 330\\
Ny Munkegade 118\\
8000 Aarhus C, Denmark
}
\email{c.spotti@qgm.au.dk}



\date{\today}

\begin{document}

\begin{abstract}
We study some basic properties and examples of Hermitian metrics on complex manifolds whose traces of the curvature of the Chern connection are proportional to the metric itself.
\end{abstract}

\maketitle

\section*{Introduction}
In this note we aim to study certain ``special'' Hermitian metrics $\omega$ on compact complex (mostly non-K\"ahler) manifolds $X$, focusing on the geometry of the Chern connection. In a previous paper \cite{angella-calamai-spotti-1} we started by looking at constant Chern-scalar curvature metrics in a conformal class of Hermitian metrics, giving partial results towards what we called  {\em Chern-Yamabe problem}.

Here instead we collect some results on the {\em Chern-Einstein problem(s)}. Namely, we look for Hermitian metrics whose Chern-Ricci forms are (non-necessarily constant) multiple of the metric itself. Note that, due to the lack of Bianchi symmetry, we have actually three different ways to contract the curvature.

A partial account on the literature includes \cite{goldberg, gauduchon-cras1980, gauduchon-cras1981, gauduchon-ivanov, catenacci-marzuoli, balas, tosatti, streets-tian, liu-yang-2, podesta, dellavedova, alekseevsky-podesta, dellavedova-gatti}, see also the discussion at \cite{mathoverflow-tosatti-ustinovskiy}. We now survey our main results.

\subsection*{First-Chern-Einstein metrics}
The {\em first-Chern-Ricci form} $\mathrm{Ric}^{(1)}(\omega)$ arises by tracing the endomorphism part of the Chern-curvature tensor. It yields a closed real $(1,1)$-form that represents the first Bott-Chern class $c_1^{BC}(X) \in H^{1,1}_{BC}(X)$.  We claim that the corresponding Chern-Einstein problem, {\itshape i.e.},  the search for Hermitian metrics $\omega$ satisfying
$$ \mathrm{Ric}^{(1)}(\omega) = \lambda \, \omega \qquad \text{ for } \lambda \in \mathcal{C}^\infty(X;\mathbb R) , $$
is basically understood. If $c_1^{BC}(X)=0$ there exists a Chern-Ricci-flat metric in any conformal class, and $X$ is called a {\em (non-K\"ahler) Calabi-Yau manifold} \cite{tosatti}, and non-K\"ahler examples actually do exist, {\itshape e.g.} Kodaira surfaces. If instead we are looking  to not identically zero (non-necessarily constant, otherwise the statement is completely trivial) Einstein factor $\lambda$, it turns out that the manifold is actually K\"ahler of a very special type, and all first-Chern-Einstein metrics are easy to describe.

\begin{thmm}[see Theorem \ref{1ChE}]
Let $(X^n,g)$ be a compact first-Chern-Einstein manifold, with non-identically-zero Einstein factor. Then $g$ is conformal to a  K\"ahler metric in $\pm c_1(X)$, with conformal factor depending explicitly on the Ricci potential of the K\"ahler metric. Conversely, starting with a  K\"ahler metric in $\pm c_1(X)$, one can always construct a Chern-Einstein metric, unique up to scaling, in the conformal class.
\end{thmm}

The above result is essentially telling that the first-Chern-Einstein problem is uninteresting for further investigations beside the fundamental K\"ahler-Einstein case or the non-K\"ahler Calabi-Yau situation \cite{szekelyhidi-tosatti-weinkove}. So let us discuss the more promising second-Chern-Einstein problem.

\subsection*{Second-Chern-Einstein metrics}
The {\em second-Chern-Ricci form} $\mathrm{Ric}^{(2)}(\omega)$ is obtained by tracing the two other indices of the Chern-curvature tensor by means of the metric. Thus a second-Chern-Einstein metric \cite{gauduchon-cras1980} is just a Hermitian metric on $TX$ that is Hermitian-Einstein by taking trace with itself.
A straighforward computation for the conformal change of the second-Chern-Ricci form shows  the second-Chern-Einstein problem depends only on the conformal class \cite{gauduchon-cras1980}. In particular, we can apply conformal methods to study the Einstein factor. Here we distinguish  between {\em weak-}second-Chern-Einstein and {\em strong-}second-Chern-Einstein metrics, according to the Einstein factor being a function, respectively, constant.

\begin{thmm}[see Theorem \ref{prop:2chei}, after {\cite{gauduchon-cras1981}, \cite[Theorems 3.1-4.1]{angella-calamai-spotti-1}}]
 Let $(X^n,g)$ be a compact weak-second-Chern-Einstein manifold. Then we can choose a representative in the conformal class of $g$ such that its Einstein factor has a sign, equal to the sign of the degree of the anti-canonical line bundle with respect to the  conformal Gauduchon metric. 
Moreover, if this sign is non-positive, then there is representative in the conformal class which is actually \emph{strong}-second-Chern-Einstein.
\end{thmm}

The assumption on the sign in the previous statement holds when the Kodaira dimension of $X$  is greater than or equal to zero, and  it would  be removed if one can prove the Chern-Yamabe conjecture \cite[Conjecture 2.1]{angella-calamai-spotti-1}. Thus, in this latter case, the weak-second-Chern-Einstein and strong-second-Chern-Einstein problems would became fully equivalent.

As a consequence, on compact manifolds, one gets a series of obstructions \cite{teleman-mathann, gauduchon-cras1981, gauduchon-bullSMF, yang-1}:
\begin{itemize}
\item the possible signs of the Einstein factor are determined according to $K_X$ and $K_X^{-1}$ being pseudo-effective or unitary flat (see Theorem \ref{thm:sign-gamma});
\item second-Chern-Ricci-flat metrics have $\mathrm{Kod}(X)\leq0$ (see Corollary \ref{cor:obstr-2-chern-flat});
\item positive second-Chern-Einstein manifolds do not have non-trivial holomorphic $p$-forms; negative second-Chern-Einstein manifolds do not have non-trivial holomorphic $p$-vector-fields, for $p\geq1$ (see Theorem \ref{thm:obstruction-2-chei}).
\end{itemize}

We also noticed that second-Chern-Ricci metrics $g$ are weakly-$g$-Hermitian-Einstein \cite{kobayashi-nagoya, lubke-teleman, kobayashi-book}, whence we get further obstructions as:
\begin{itemize}
\item the Bogomolov-L\"ubcke inequality \eqref{eq:bogomolov-lubke} holds \cite{bogomolov, lubke};
\item the Kobayashi-Hitchin correspondence assures that the holomorphic tangent bundle is $g$-semi-stable \cite{kobayashi, lubke-man}.
\end{itemize}

We conclude by providing several examples.
 Clearly holomorphically-parallelizable manifolds are Chern-flat, hence strong-second-Chern-Einstein with \emph{zero} Einstein factor. See \cite{boothby} for a characterization of compact Hermitian Chern-flat manifolds as quotients of complex Hermitian Lie groups. Furthermore, the classical Hopf manifold and the new examples by Fabio Podest\`a \cite{podesta} are strong-second-Chern-Einstein, both with \emph{positive} Einstein factor.
On the other hand, in the literature we did not find any example of  non-K\"ahler second-Chern-Einstein metric with {\em negative} Einstein factor, even local. The following result provide such examples on (non-compact) four-dimensional solvable Lie groups with invariant complex structures, as classified in \cite{snow-crelle, snow, ovando-manuscr, achk}, see \cite{ovando}. Explicit computations are performed with the help of Sage \cite{sage}.

\begin{thmm}[see Section \ref{subsubsec:ovando-negative}]
There are four-dimensional solvable Lie groups endowed with invariant complex structures  admitting strong-second-Chern-Einstein metrics with \emph{negative} Einstein factors.
\end{thmm}

We can perform explicit computations on compact quotients of four-dimensional solvable Lie groups \cite{hasegawa-jsg}, including in particular the non-K\"ahler examples of Inoue surfaces and Kodaira surfaces. As regards strong-second-Chern-Einstein metrics on these manifolds, we have the following:

\begin{thmm}[see Theorem \ref{thm:solv-surfaces}]
Inoue surfaces and Kodaira surfaces, seen as quotients of solvable Lie groups endowed with invariant complex structures, do not admit any invariant second-Chern-Einstein metric.
On the other hand, Kodaira surfaces admit first-Chern-Ricci-flat metrics.
\end{thmm}

\begin{rmk}
In fact, as we learnt by Stefan Ivanov, second-Chern-Einstein metrics can not exist on non-K\"ahler compact complex surfaces, except for the Hopf surfaces, thanks to \cite[Theorem 2]{gauduchon-ivanov}, see also \cite{gauduchon-cras1980}.
We ask whether {\em there is any non-K\"ahler example of negative second-Chern-Einstein metric on a compact complex manifold} of higher dimension.
\end{rmk}

\begin{rmk}
We have not investigated the third way to contract the Chern curvature yet, since we are very skeptical about the possible geometric meaning of such quantity.
\end{rmk}

\bigskip

\begin{small}
\noindent{\sl Acknowledgments.}
The authors are grateful to Stefan Ivanov, Fabio Podestà, Valentino Tosatti, Yury Ustinovskiy for several useful discussions on the subject. During the preparation of this note, the first two named authors have been supported by the Project PRIN ``Varietà reali e complesse: geometria, topologia e analisi armonica'', by the Project SIR2014 "Analytic aspects in complex and hypercomplex geometry" (AnHyC) code RBSI14DYEB, and by GNSAGA of INdAM. DA is further supported by project FIRB ``Geometria Differenziale e Teoria Geometrica delle Funzioni''', and  SC  by  the Simons Center for Geometry and Physics, Stony Brook University.  The third-named author is supported by AUFF Starting Grant 24285,  Villum Young Investigator 0019098, and DNRF95 QGM ``Centre for Quantum Geometry of Moduli Spaces''.
\end{small}

\section{Preliminaries and notation}
Let $X^n$ be a complex $n$-dimensional manifold endowed with a Hermitian structure $h=g-\sqrt{-1}\omega$, where $g$ is the associated Riemannian metric and $\omega = g(\_,J\_)$ is the associated $(1,1)$-form; here, we consider $T^*X$ endowed with the dual complex structure $J\alpha:=\alpha(J^{-1}\_)$.
The holomorphic tangent bundle $T_X$ is then a holomorphic Hermitian bundle, and we consider its Chern connection $\nabla^{Ch}$, namely, the unique Hermitian connection on $T_X$ extending the Cauchy-Riemann operator $\bar\partial$.
We denote by $\Theta := (\nabla^{Ch})^2 \in \wedge^2(X;\mathrm{End}\,T_X) \stackrel{g}{\simeq} \wedge^2(X;T_X\otimes \overline{T_X})$ the curvature tensor of the Chern connection. In local holomorphic coordinates $z^j$, denoting with $h_{i\bar{j}}$ the Hermitian matrix representing the Hermitian structure $h$, the curvature $\Theta(\omega)$ has the following expression:
$$ \Theta(\omega)= \Theta_{i\bar{j}k\bar{\ell}} \sqrt{-1} dz^i \wedge d\bar z^j \otimes \sqrt{-1} dz^k \wedge d\bar z^\ell $$
where
$$ \Theta_{i\bar{j}k\bar{\ell}} = - \frac{\partial^2 h_{k\bar{\ell}}}{\partial z^i \partial \bar z^j} + h^{p\bar{q} } \frac{\partial h_{k \bar{q}}}{\partial z^i}\frac{\partial h_{p \bar{\ell}}}{\partial \bar z^j}.$$

\subsection{Chern-Ricci curvatures and Chern-Einstein metrics}

There are   essentially three, in general different, ways to contract the above quantity $\Theta(\omega)$.
We can define two real $(1,1)$-forms as follows, \cite[Section I.4]{gauduchon-mathann}, see also \cite[Section 3.1.2]{liu-yang-2}. In local coordinates, the {\em first Chern-Ricci form} is
\begin{eqnarray*}
\mathrm{Ric}^{(1)}(\omega) &:=& \mathrm{tr}\, \Theta(\omega) = \mathrm{tr}\,\Theta_{i\bar j \bullet}^{\phantom{i\bar j\bullet}\bullet} \sqrt{-1} dz^i \wedge d\bar z^j \\
&=& h^{k\bar{\ell}} \Theta_{i\bar{j}k\bar{\ell}} \sqrt{-1} dz^i \wedge d\bar z^j = -\frac{\partial^2\log\det \, (h_{k\bar\ell}) \, }{\partial z^i \partial\bar z^j} dz^i\wedge d\bar z^j ,
\end{eqnarray*}
and the {\em second Chern-Ricci form} is
$$
\mathrm{Ric}^{(2)}(\omega) =  \mathrm{tr}_g \Theta_{\bullet\bullet k\bar \ell} \sqrt{-1} dz^k \wedge d\bar z^\ell
= h^{i\bar{j}} \Theta_{i\bar{j}k\bar{\ell}} \sqrt{-1} dz^k \wedge d\bar z^\ell .
$$
Note that $\mathrm{Ric}^{(1)}(\omega)$ is a closed $(1,1)$-form that represents the first Chern class $c_1(X):=c_1(T_X)=c_1(K_X^{-1})\in H^2(X;\R)$. More precisely, it represents the first Bott-Chern class $c_1^{BC}(X)\in H^{1,1}_{BC}(X)$. On the other hand, $\mathrm{Ric}^{(2)}(\omega)$ is in general not closed.

We also notice that there is a third way to contract the curvature $\Theta$, namely, define the tensor
$$ \mathrm{Ric}^{(3)}_{k\bar{j}}(\omega) = \mathrm{tr}_g \Theta_{\bullet \bar j k \bullet} = h^{i\bar{\ell}} \Theta_{i\bar{j}k\bar{\ell}} = \overline{ h^{i\bar{\ell}} \Theta_{j\bar{\ell}i\bar{k}} }. $$
On a K\"ahler manifold, (or more in general for Hermitian metrics being K\"ahler-like in the sense of \cite{yang-zheng},) the three Chern-Ricci curvatures coincide.

\begin{defi}
For $i\in \{1,2\}$ a Hermitian manifold $(X^n,g)$ is called \emph{(i)-Chern-Einstein} if
$$\mathrm{Ric}^{(i)}(\omega)= \lambda \,  \omega,$$
for some real-valued \emph{function} $\lambda \in C^{\infty}(X;\R)$, which is called the {\em Einstein factor}, where $\mathrm{Ric}^{(i)}(\omega)$ denotes the $(i)$-contraction as above and $\omega$ denotes the natural defined $(1,1)$-form associated to the metric.
\end{defi}
Sometimes we distinguish between {\em weak-(i)-Chern-Einstein} and {\em strong-(i)-Chern-Einstein} metrics, according to the Einstein factor being a function, respectively, constant.

The {\em Chern-scalar curvature} is defined as
$$ S^{Ch}(\omega) := \mathrm{tr}_g \mathrm{Ric}^{(1)} = \mathrm{tr}_g\, \mathrm{Ric}^{(2)} =  h^{i \bar j} h^{k\bar{\ell}} \Theta_{i\bar{j}k\bar{\ell}} , $$
and the {\em third Chern-scalar curvature} is defined as
$$ S^{(3)}(\omega) := h^{k \bar j} h^{i\bar{\ell}} \Theta_{i\bar{j}k\bar{\ell}} . $$
(See \cite[Equation (4.18)]{liu-yang-2} for a comparison between $S^{Ch}$ and $S^{(3)}$.)
Clearly, if $\omega$ is either (1)-Chern-Einstein or (2)-Chern-Einstein with Einstein factor $\lambda$, then $S^{Ch}(\omega)=n\lambda$.

\begin{rmk}
Compare \cite{goldberg-PAMS, draghici, alekseevsky-podesta, dellavedova, lejmi-upmeier, dellavedova-gatti} in the almost-complex setting.
\end{rmk}

\subsection{Formulas for the conformal changes}
Let $\omega_f:=e^f\omega$ be a conformal change of $\omega$. Then we compute for the curvature form $\Theta_f:=\Theta(e^f\omega)$ of the Chern connection for such metric:
$$(\Theta_f)_{i\bar{j}k\bar{\ell}}=e^f( \Theta_{i\bar{j}k\bar{\ell}}-h_{k\bar{\ell}} \partial^2_{i\bar{j}}f).$$
Thus: 
\begin{eqnarray}
\label{eq:conf-1}
\mathrm{Ric}^{(1)}(\omega_f) &=& \mathrm{Ric}^{(1)}(\omega)- n \sqrt{-1} \partial \bar \partial f , \\
\label{eq:conf-2}
\mathrm{Ric}^{(2)}(\omega_f) &=& \mathrm{Ric}^{(2)}(\omega)-  \omega \, \Delta^{Ch} f ,
\end{eqnarray}
where $\Delta^{Ch}$ denotes the Chern-Laplacian with respect to $\omega$, namely, $\Delta^{Ch}f := \frac{1}{2} \left( \omega, dd^c f \right)_\omega = \Delta_{\omega} f+( d  f,\vartheta)_\omega = -2h^{j\bar k}\partial^2_{j \bar k} f$, where $\vartheta$ is defined by $d\omega^{n-1}=:\vartheta\wedge\omega^{n-1}$. See \cite[Corollary 4.4]{lejmi-upmeier} in the more general almost-Hermitian setting.

\section{First-Chern-Einstein metrics}
Let $X^n$ be a compact complex manifold endowed with a Hermitian metric $\omega$. Consider the Chern connection $\nabla^{Ch}$, and the (1)-Chern-Ricci form. Denote by $T(X,Y):=\nabla^{Ch}_XY-\nabla^{Ch}_YX-[X,Y]$ the torsion tensor of $\nabla^{Ch}$, and by $\tau \stackrel{\text{loc}}{:=} T_{jk}^{k} dz^j$ its trace-of-torsion form.
In this section, we investigate the condition weak-(1)-Chern-Einstein, namely, when $\mathrm{Ric}^{(1)}(\omega)=\lambda \omega$ for $\lambda \in \mathcal C^\infty(X;\mathbb R)$.	

First of all, we notice that the Einstein factor $\lambda$ is constant non-zero if and only if the torsion form $\tau$ vanishes identically, by \cite[Theorem 4]{goldberg}. In this case, the metric is actually K\"ahler, since $d\mathrm{Ric}^{(1)}(\omega)=0$.
(The same conclusion by \cite{goldberg} holds true for the strong-(3)-Chern-Ricci curvature, see \cite[Theorem 4.1]{balas}: namely, (3)-Chern-Einstein metrics with constant non-zero Einstein factor are actually K\"ahler.)

A way to construct weak-(1)-Chern-Einstein metrics is via the following completely elementary trick based on the $\sqrt{-1}\partial\bar{\partial}$-lemma. If $X^n$ is compact and its first Chern class $c_1(X)$ has a sign, {\itshape i.e.}, there exists a K\"ahler metric $\omega \in \pm 2 \pi c_1(X) = \pm 2 \pi c_1^{BC}(X) \in H^{1,1}_{BC}(X)$, then there exists a function $f\in C^\infty(X;\R)$ (the Ricci potential), unique up to a constant, such that $\mathrm{Ric}^{(1)}(\omega)=\pm\omega +  \sqrt{-1}\partial\bar{\partial} f$.
Consider the conformal metric $\omega_f:= e^{\frac{f}{n}}\omega$. Then, thanks to \eqref{eq:conf-1},
$$\mathrm{Ric}^{(1)}(\omega_f)=\mathrm{Ric}^{(1)}(\omega)-\sqrt{-1}\partial\bar{\partial} f=\pm \omega= \pm e^{-\frac{f}{n}} \omega_f.$$
Note that if we instead assume $c_{1}^{BC}(X)=0$, the same argument gives a weak-(1)-Chern-Ricci-flat metric, unique up to scaling in its conformal class. 

Our first result shows that \emph{all} weak-(1)-Chern-Einstein metrics with non-zero Einstein factor are constructed in the above way.
Therefore, in non-K\"ahler geometry, the only interesting case is for (1)-Chern-Einstein-flat metrics. Note that our result has a flavor similar to the picture described in \cite{lebrun, chen-lebrun-weber} in the case of complex surfaces.

\begin{thm}\label{1ChE}
Let $(X^n,g)$ be a compact weak-(1)-Chern-Einstein manifold of dimension $n\geq 2$, and suppose that the Einstein factor $\lambda$ is not identically equal to zero. Then:
\begin{itemize}
 \item $X$ is Fano or anti-Fano, {\itshape i.e.} the first Chern class $c_1(X)$ has a sign;
 \item $g$ is conformally equivalent to a K\"ahler metric $\eta \in \pm 2 \pi c_1(X)$;
 \item up to scaling, $\lambda$ must be given by $\pm e^{-\frac{f}{n}}$,  where $f$ is a Ricci potential for $\eta$.
\end{itemize}
In particular, such metrics are unique up to scaling in their conformal class.
\end{thm}

\begin{proof}
 Let $\omega$ be a weak-(1)-Chern-Einstein metric with Einstein factor $\lambda$ not identically equal to zero, and let $\eta$ be the unique Gauduchon representative of volume one in its conformal class \cite[Théorème 1]{gauduchon-cras1977}. Thus $\omega =e^{g}\eta $ for some function $g$. 
 Define the function $\tilde \lambda:= \lambda \,  e^{g} $ and consider the $(n-1,n)$-form
 $$ \tilde \lambda^{n-1} \sqrt{-1}\, \overline\partial \tilde \lambda^{n-1}\wedge \eta^{n-1}.$$
 Note that, by the Chern-Einstein condition, $$\partial ( \tilde \lambda \eta)=  \partial (\lambda \omega)= \partial (\mathrm{Ric}^{(1)}(\omega))=0.$$
 Hence $\tilde \lambda \partial \eta =-\partial \tilde \lambda \wedge \eta$.
 Now we compute
 \begin{eqnarray*}
  \lefteqn{ \tilde \lambda^{n-1} \sqrt{-1}\, \partial\overline\partial \tilde \lambda^{n-1}\wedge \eta^{n-1} } \\[5pt]
&=& d\left(\tilde \lambda^{n-1} \sqrt{-1}\, \overline\partial \tilde \lambda^{n-1}\wedge
\eta^{n-1} \right) - \partial\tilde \lambda^{n-1}\wedge \sqrt{-1}\, \overline\partial\tilde
\lambda^{n-1}\wedge\eta^{n-1}\\[5pt]
&&+\tilde \lambda^{n-1} \sqrt{-1}\, \overline\partial\tilde
\lambda^{n-1}\wedge\partial\eta^{n-1} \\[5pt]
  &=& d\left(\tilde \lambda^{n-1} \sqrt{-1}\, \overline\partial \tilde \lambda^{n-1}\wedge
\eta^{n-1} \right) .
 \end{eqnarray*}

Recall that the Hodge-de Rham Laplacian $\Delta_{d,\eta}=[d,d^*_\eta]$ and the Chern Laplacian $\Delta^{Ch}_\eta=2\sqrt{-1}\,\mathrm{tr}_\eta\overline\partial\partial$ on smooth functions are related by \cite[pages 502--503]{gauduchon-mathann}:
$$ \Delta^{Ch}_\eta f = \Delta_{d,\eta} f + (df,\, \theta_\eta )_\eta , $$
where $\theta_\eta$ denotes the co-closed torsion $1$-form of the Gauduchon metric, that is, $d\eta^{n-1}=\theta_\eta\wedge\eta^{n-1}$, and $d^*_\eta\theta_\eta=0$. 
Integrating on the compact manifold $X$, we get:
 \begin{eqnarray*}
  0 &=& \int_X \tilde \lambda^{n-1} \sqrt{-1}\, \partial\overline\partial \tilde \lambda^{n-1}\wedge
\eta^{n-1} 
  = \frac{1}{2n} \int_X \tilde \lambda^{n-1} \Delta^{Ch}_\eta \tilde \lambda^{n-1} \eta^n \\[5pt]
  &=& \frac{1}{2n} \int_X \tilde \lambda^{n-1} \Delta_{d,\eta} \tilde \lambda^{n-1} \eta^n
  + \frac{1}{2n} \int_X \tilde \lambda^{n-1} \left( d \tilde \lambda^{n-1}, \theta_\eta \right)_\eta \eta^n\\[5pt]
  &=& \frac{1}{2n} \int_X \left| d \tilde \lambda^{n-1}\right|^2 \eta^n + \frac{1}{4n} \int_X \left( d \tilde \lambda^{2n-2}, \theta_\eta \right)_\eta \eta^n \\[5pt]
  &=& \frac{1}{2n} \int_X \left| d \tilde \lambda^{n-1}\right|^2 \eta^n .
  \end{eqnarray*}
Thus our function $\tilde \lambda $ is a non-zero constant and $\eta$ is in fact K\"ahler.

Using again the Chern-Einstein condition we find
$$ 2 \pi c_1(X) = 2 \pi c_1^{BC}(X) \ni \mathrm{Ric}^{(1)}(\eta) = \tilde \lambda  \eta +  \sqrt{-1}\, \partial\bar\partial (ng). $$
	 Hence the first Chern class $c_1(X)\lessgtr 0$ depending on the value of $\tilde \lambda $, {\itshape i.e.}, $X$ is Fano or anti-Fano. Thus our original weak-(1)-Chern-Einstein metric is conformally K\"ahler, and the Einstein factor $\lambda$ is given, up to scaling,  by the Ricci potential of a K\"ahler metric in $\pm 2\pi  c_1(X)$. Uniqueness up to scaling of weak-(1)-Chern-Einstein metrics in their conformal class follows immediately by the uniqueness up to constant of the Ricci potential.
\end{proof}

In particular, (1)-Chern-Einstein metrics on non-K\"ahler manifolds occur only for (1)-Chern-Ricci-flat metrics. In particular, $c_1^{BC}(X)=0$, which also implies $c_1(X)=0$, (but not the converse: compare {\itshape e.g.}, the Hopf manifold in Example \ref{ex:hopf}).
In this case, the following result gives a complete description:

\begin{thm}[see {\cite[Proposition 1.1, Theorem 1.2]{tosatti}, \cite[Corollary 2]{tosatti-weinkove}}]
On a compact Hermitian manifold $(X^n,g)$ with $c_1^{BC}(X)=0$, there exists a strong-(1)-Chern-Einstein metric with Einstein factor $\lambda=0$. It can be given either as a conformal transformation $e^\varphi g$ of $g$; or as associated to the form $\omega+\sqrt{-1}\,\partial\overline\partial \varphi$ where $\omega$ is the fundamental form of $g$. In both cases, $\varphi$ is unique up to additive constants.
\end{thm}

In \cite{szekelyhidi-tosatti-weinkove}, see also \cite{popovici}, by proving the Gauduchon's generalization of the Calabi's conjecture for compact complex manifolds satisfying $c_1^{BC}(X)=0$, it is proved that it is also always possible to find \emph{Gauduchon} (1)-Chern-Ricci-flat metrics, hence providing existence of non-K\"ahler special metrics satisfying both curvature and cohomological conditions.

\section{Second-Chern-Einstein metrics}

We consider now the second Chern-Ricci curvature, and the corresponding Chern-Einstein condition.

\subsection{Second-Chern-Einstein metrics in the conformal class}
The first fundamental remark is that, thanks to \eqref{eq:conf-2}, the weak-(2)-Chern-Einstein property is a property of the conformal Hermitian structure, and not just of the Hermitian structure, see also \cite{gauduchon-cras1980}. More precisely, under the conformal change $\omega \mapsto \exp(-f)\cdot\omega$, the Einstein factor changes as $\lambda \mapsto \exp(f) \cdot \left( \lambda + \Delta^{Ch}_{\omega} f \right)$.
In particular, thanks to the Gauduchon conformal methods \cite{gauduchon-cras1981}, (see {\itshape e.g.}, the preliminary step in the Proof of Theorem 4.1 in \cite{angella-calamai-spotti-1},) we can assume that $\lambda$ has a sign without loss of generality.

The sign is determined by an invariant of the conformal class as follows. In the conformal class $\{\omega\}$, choose the unique Gauduchon metric $\eta=e^f\omega$ with unitary volume, thanks to \cite[Th\'eor\`eme 1]{gauduchon-cras1981}. Define the {\em Gauduchon degree} of $\{\omega\}$ as the degree of the anti-canonical line bundle $K_X^{-1}$ with respect to the Gauduchon metric $\eta$, namely,
$$ \Gamma_X(\{\omega\}) := \int_X c_1^{BC}(K_X^{-1}) \wedge \frac{1}{(n-1)!} \eta^{n-1} = \int_X S^{Ch}(\eta) \eta^n . $$

Note indeed that, if $\eta$ is weak-(2)-Chern-Einstein with Einstein factor $\lambda_\eta$, so that $\omega$ is weak-(2)-Chern-Einstein with Einstein factor $\lambda_\omega=e^f(\lambda_\eta+\Delta_\eta^{Ch}f)$, then $S^{Ch}(\eta)=n\lambda_\eta=ne^{-f}\lambda_\omega-\Delta^{Ch}_\eta f$, where $\int_X \Delta^{Ch}_\eta f \eta^n=0$.
In particular, if $\mathrm{Kod}\,X \geq 0$, (respectively, $\mathrm{Kod}\,X > 0$,) then by \cite{gauduchon-cras1981}, see also \cite[\S I.17]{gauduchon-mathann}, we get that $\Gamma_X(\{\eta\})\leq 0$, (respectively, $\Gamma_X(\{\eta\}) < 0$,) for any conformal class $\{\eta\}$.

In the notation as above, the problem of finding a strong-(2)-Chern-Einstein metric in $\{\omega\}$ with {\em constant} Einstein factor, reduces to solve the Liouville-type equation
\begin{equation}\label{eq:chya}
\Delta_\eta^{Ch}f + \frac{1}{n} S^{Ch}(\eta) = \lambda e^{-f}
\end{equation}
for $(f,\lambda)\in \mathcal{C}^\infty(X;\mathbb R)\times \mathbb R$. We introduced and investigated this equation \eqref{eq:chya} in \cite{angella-calamai-spotti-1} under the name of {\em Chern-Yamabe equation}.
In particular, we proved \cite[Theorems 3.1-4.1]{angella-calamai-spotti-1} that, if $\Gamma_X(\{\omega\})$ is non-positive, then the Chern-Yamabe equation admits a unique solution up to scaling. Summarizing:

\begin{thm}[{\cite{gauduchon-cras1981}, \cite[Theorems 3.1-4.1]{angella-calamai-spotti-1}}]\label{prop:2chei}
 Let $(X^n,g)$ be a compact Hermitian manifold, and assume that the conformal class of $g$ is weak-(2)-Chern-Einstein. Then we can choose a representative in the conformal class of $g$ such that the Einstein factor has a sign, equal to the sign of the Gauduchon degree $\Gamma_X(\{g\})$.
 
Moreover, if $\Gamma_X(\{g\})\leq 0$, (for example, if $\mathrm{Kod}\,X\geq 0$,) then we can choose a representative in the conformal class having non-positive constant Einstein factor.
If one can prove the Chern-Yamabe conjecture \cite[Conjecture 2.1]{angella-calamai-spotti-1}, then the same holds true without any assumption on the sign of the Gauduchon degree.
\end{thm}

Note in particular that the existence of a constant positive Einstein factor does not force its uniqueness in the conformal class, compare \cite[Section 5.5]{angella-calamai-spotti-1}.

The possible values for the sign of weak-(2)-Chern-Einstein metrics are summarized in the following:
\begin{thm}[{\cite{teleman-mathann}, \cite[Theorem 1.1, Theorem 3.4]{yang-1}}]\label{thm:sign-gamma}
Let $X^n$ be a compact complex manifold. We look at the image of the application $\Gamma_X$ that associates to each Hermitian conformal class $\{\eta\}\in \mathcal{H}erm\mathcal{C}onf(X)$, its Gauduchon degree $\Gamma_X(\{\eta\})\in\R$:
 \begin{itemize}
 \item $\Gamma_X(\mathcal{H}erm\mathcal{C}onf(X))=\mathbb{R}$ if and only if neither $K_X$ nor $K_X^{-1}$ is pseudo-effective;
 \item $\Gamma_X(\mathcal{H}erm\mathcal{C}onf(X))=\mathbb{R}^{>0}$ if and only if $K_X^{-1}$ is pseudo-effective and non-unitary-flat;
 \item $\Gamma_X(\mathcal{H}erm\mathcal{C}onf(X))=\mathbb{R}^{<0}$ if and only if $K_X$ is pseudo-effective and non-unitary-flat;
 \item $\Gamma_X(\mathcal{H}erm\mathcal{C}onf(X))=\{0\}$ if and only if $K_X$ is unitary-flat.
 \end{itemize}
\end{thm}

We recall that a holomorphic line bundle over $(X^n,g)$ compact Hermitian manifold is called {\em pseudo-effective} if it admits a (possibly singular) Hermitian metric with non-negative curvature (in the sense of currents); it is called {\em unitary-flat} if it admits a smooth Hermitian metric with zero curvature.

In particular, one gets an obstruction for the existence of weak-(2)-Chern-Einstein-flat metrics:

\begin{cor}[{\cite{gauduchon-cras1981}, \cite[Theorem 1.4]{yang-1}}]\label{cor:obstr-2-chern-flat}
Let $X^n$ be a compact complex manifold. It admits a weak-(2)-Chern-Einstein metric with Einstein factor $\lambda=0$ only if:
\begin{itemize}
\item either: $\mathrm{Kod}\,X=-\infty$ and neither $K_X$ nor $K_X^{-1}$ is pseudo-effective; 
\item or: $\mathrm{Kod}\,X=-\infty$ and $K_X$ is unitary-flat;
\item or: $\mathrm{Kod}\,X=0$ and $K_X$ is holomorphically-torsion (namely, there exists $m\in\N$ such that $K_X^{\otimes m}$ is trivial; in particular, $K_X$ is unitary-flat).
\end{itemize}
\end{cor}

Obstructions {\itshape à la} Bochner follow by \cite{gauduchon-bullSMF}, see also \cite[Theorem at page 1]{kobayashi-wu} and, more in general, \cite[Theorem 1.1]{liu-yang-1}, give further obstructions:

\begin{thm}[{\cite[Corollaire 2 at page 124]{gauduchon-bullSMF}, \cite[Corollary 1.2]{liu-yang-1}}]\label{thm:obstruction-2-chei}
Let $X^n$ be a compact complex manifold. If it admits a weak-(2)-Chern-Einstein metric with Einstein factor $\lambda\neq0$, then:
\begin{itemize}
\item in case $\lambda \geq 0$, then $H^{p,0}_{\overline\partial}(X)=H^0(X;\wedge^pT_X^*)=0$ for any $p\geq1$. In particular, the arithmetic genus is $\chi(X;\mathcal{O}_X)=1$, by the Hirzebruch-Riemann-Roch theorem, and $X$ does not admit finite covers;
\item in case $\lambda \leq 0$, then $H^0(X;\wedge^pT_X)=0$ for any $p\geq1$. In particular, there are no non-trivial holomorphic vector fields.
\end{itemize}
As for strong-(2)-Chern-Einstein metric with Einstein factor $\lambda=0$, then any holomorphic $p$-form and any holomorphic $p$-vector field is parallel with respect to the Chen connection.
\end{thm}

By Theorem \ref{prop:2chei}, we then know that the hypothesis of the above Theorem \ref{thm:obstruction-2-chei} hold depending on the sign of the Gauduchon degree of a (2)-Chern-Einstein conformal class.

\subsection{Second-Chern-Einstein metrics as Hermitian-Einstein metrics on the tangent bundle}
Another observation is that weak-(2)-Chern-Einstein metrics $g$ yield that that the holomorphic tangent bundle $(T_X, g)$ is weakly-$g$-Hermitian-Einstein, see {\itshape e.g.} \cite{kobayashi-nagoya, lubke-teleman, kobayashi-book}.
In particular, one has the obstruction given by the Bogomolov-L\"ubke inequality \cite{bogomolov, lubke}, see {\itshape e.g.} \cite[Theorem 2.2.3]{lubke-teleman}:
\begin{equation}\label{eq:bogomolov-lubke}
\left((n - 1) c_1(X)^2 - 2 n c_2(X) \right) \wedge \omega^{n - 1} \leq 0 ,
\end{equation}
and the Kobayashi-Hitchin correspondence \cite{kobayashi-nagoya, kobayashi, lubke-man}, see {\itshape e.g.} \cite[Theorem 2.3.2]{lubke-teleman}, namely, the holomorphic tangent bundle is $g$-semi-stable.

In analogy to the K\"ahler-Einstein/K-stable case, it seems interesting to understand if there is a more refined version of \emph{stability of the complex manifold itself} (plus some extra data, such as the Aeppli class of the $(n-1)$th power of the Gauduchon representative of a conformal class) that characterizes existence of (2)-Chern-Einstein metrics.

\begin{rmk}
Projectively flat metrics on the holomorphic tangent bundle of a compact complex manifold are a special instance of second-Chern-Einstein metrics. In \cite{calamai} is given a classification of them and is proven that each of them is strong-second-Chern-Einstein; futhermore there is shown (see also \cite{matsuo}) that projectively flat metrics providing examples of negative strong-second-Chern-Einstein metrics do not exist.
\end{rmk}

\subsection{Examples of second-Chern-Einstein metrics}

In this section we concentrate on examples of Chern-Einstein metrics. We begin with a first, classical and illustrative example: 

\subsubsection{Hopf manifold}\label{ex:hopf}\label{subsubsec:hopf-positive}
The {\em Hopf surface} admits a strong-(2)-Chern-Einstein Hermitian metric with constant positive Einstein factor, see {\itshape e.g.} \cite[Example 5.1]{yang-1}, \cite[Section 6]{liu-yang-1}, \cite[Theorem 1.5]{liu-yang-2}.

Let $X^2=(\C^2\setminus 0) \slash \Gamma$ be a Hopf surface, namely, a compact complex surface with universal cover $\C^2\setminus0$ and fundamental group isomorphic to $\Z$. We have $\mathrm{Kod}\,X=-\infty$. The anti-canonical bundle $K_X^{-1}$ is pseudo-effective, while the canonical bundle $K_X$ is not pseudo-effective. In particular, the Gauduchon degree is positive, $\Gamma_X(\{\eta\})>0$, for any conformal class $\{\eta\}$. The first Bott-Chern class is not zero, $c_1^{BC}(X)\neq0$, while its image in de Rham cohomology is zero, $c_1(X)=0$. We have that $h^{p,0}(X)=0$ and $H^0(X;\wedge^pT_X)=0$ for any $p\geq1$, and that the holomorphic tangent bundle $T_X$ is semi-stable.

Thanks to Theorem \ref{1ChE}, the manifold $X$ does not admit any (1)-Chern-Einstein metric, because $c_1^{BC}(X)\neq0$ and it is not K\"ahler. 

On the other hand, consider Hopf surfaces with $\Gamma$ generated by $\left(\begin{matrix}\alpha&0\\0&\beta\end{matrix}\right)$ where $0<|\alpha|=|\beta|<1$. Consider the standard locally conformally K\"ahler Hermitian metric induced by  a conformal change of the flat metric on $\C^2$ as below with $r=1$: it is straightforward to see that it is strong-(2)-Chern-Einstein with constant positive Einstein factor $\lambda=2$ ({\itshape e.g.} \cite[Section 6]{liu-yang-1}). More concretely, looking at the diffeomorphism type $X \simeq \mathbb S^1 \times \mathrm{SU}(2)$, the manifold $X$ is described by a coframe of global $(1,0)$-forms $\{\varphi^1,\varphi^2\}$ with structure equations
$$ d\varphi^1 = \sqrt{-1}\,\varphi^1\wedge\varphi^2+\sqrt{-1}\,\varphi^1\wedge\bar\varphi^2, \qquad d\varphi^2=-\sqrt{-1}\,\varphi^1\wedge\bar\varphi^1 .$$
For the invariant metric with fundamental form
$$ \omega =  \frac{\sqrt{-1}}{2}\, r^2\, \varphi^1\wedge\bar\varphi^1 + \frac{\sqrt{-1}}{2}\, r^2\,\varphi^2\wedge\bar\varphi^2 , $$
where $r\in\mathbb R \setminus\{0\}$,
we compute the Chern curvature: the only non-zero components are
$$
\Theta_{1\bar11\bar1} = \frac{1}{2}\,r^{2}, \qquad
\Theta_{1\bar12\bar2} = -\frac{1}{2}\,r^{2},
$$
and the ones corresponding to the symmetries.
In particular, the first-Chern-Ricci form is
$$ \mathrm{Ric}^{(1)} = 2\,\sqrt{-1}\, \varphi^1\wedge\bar\varphi^1 $$
and the second-Chern-Ricci form is
$$ \mathrm{Ric}^{(2)} = \sqrt{-1}\, \varphi^1\wedge\bar\varphi^1 + \sqrt{-1}\, \varphi^2\wedge\bar\varphi^2 = \frac{2}{r^2} \omega. $$
namely, the standard metric $\omega$ is strong-(2)-Chern-Einstein with constant Einstein factor $\lambda=\frac{2}{r^2}$.
The Chern-scalar curvature is then $S^{Ch}=\frac{4}{r^{2}}$.

We also notice that the third Chern-Ricci tensor has the only non-zero component $\mathrm{Ric}^{(3)}_{1\bar1}=1$, and it gives the same third Chern-scalar curvauture $S^{(3)}=S^{Ch}$.

\subsubsection{Podest\`a examples}\label{subsubsec:podesta}
More examples of strong-(2)-Chern-Einstein metrics with constant positive Einstein factor are provided by F. Podest\`a in \cite{podesta} among $\mathcal C$-manifolds.
They are compact simply-connected complex homogeneous spaces, given by the product of two $G_j$-homogeneous spaces of the form $G_j / L_j$ where $L_j$ is a connected subgroup of $G_j$ that coincides with the semisimple part of the centralizer of a torus in $G_j$, for $j\in\{1,2\}$. Manifolds in such class are $\mathbb T^2$-bundle over the product of two compact irreducible Hermitian symmetric spaces, and can be endowed with a two-parameter family of inequivalent invariant complex structures, which do not admit either K\"ahler or balanced metrics. Examples are given by the Calabi-Eckmann manifolds. In \cite[Theorem 3]{podesta}, it is shown that, for any such invariant complex structure, there exists an invariant Hermitian metric being strong-(2)-Chern-Einstein with Einstein factor equal to $1$.

\subsubsection{Compact complex surfaces diffeomorphic to solvmanifolds}\label{subsubsec:surfaces}
Other than the K\"ahler-Einstein case, compact complex surfaces being first-Chern-Einstein are completely understood thanks to the condition $c_1^{BC}(X)=0$: they include Kodaira surfaces beside Calabi-Yau surfaces (namely, complex $2$-dimensional tori, Enriques surfaces, bi-elliptic surfaces, K3 surfaces).
In \cite[Theorem 2]{gauduchon-ivanov}, it is proven that the only non-K\"ahler compact complex surfaces admitting second-Chern-Einstein metrics are the Hopf surfaces. Here, we perform explicit computations on compact complex (non-K\"ahler) surfaces diffeomorphic to solvmanifolds, namely, compact quotients of solvable Lie groups, according to \cite[Theorem 1]{hasegawa-jsg}. Besides the K\"ahler case (complex torus, hyperelliptic surface), we have Inoue surfaces and Kodaira surfaces, endowed with invariant complex structures. We consider invariant Hermitian metrics: in the notation above (we recall that $g=\omega(J\_,\_)$, and $J\alpha=\alpha(J^{-1}\_)$ on the dual), with respect to a chosen coframe $\{\varphi^1,\varphi^2\}$ of invariant $(1,0)$-forms, invariant Hermitian metrics are given by
\begin{equation}\label{eq:metric}
\omega = \frac{\sqrt{-1}}{2} \, r^{2} \varphi^{1}\wedge\bar\varphi^{1} + \frac{\sqrt{-1}}{2} \, s^{2} \varphi^{2}\wedge\bar\varphi^{2} + \frac{1}{2} \, u  \varphi^{1}\wedge\bar\varphi^{2} - \frac{1}{2} \, \overline{u} \varphi^{2}\wedge\bar\varphi^{1} 
\end{equation}
varying $r,s\in\mathbb R \setminus \{0\}$ and $u\in\mathbb C$ such that $r^2s^2-|u|^2>0$.

We give here their curvature tensors, showing by explicit computations that:

\begin{thm}\label{thm:solv-surfaces}
On Inoue and Kodaira surfaces, seen as quotients of solvable Lie groups endowed with invariant complex structures, there is no invariant Hermitian metric being strong-(2)-Chern-Einstein.
\end{thm}

\begin{proof}The proof is based on a case-by-case check.
\begin{description}
\item[Inoue $S_M$]
The structure equations can be given as $d\varphi^1=\frac{\sqrt{-1}}{4}\varphi^1\wedge\varphi^2-\frac{\sqrt{-1}}{4}\varphi^1\wedge\bar\varphi^2$, $d\varphi^2=\frac{\sqrt{-1}}{2}\varphi^2\wedge\bar\varphi^2$.
The first-Chern-Ricci form is $\mathrm{Ric}^{(1)} = - \frac{\sqrt{-1}}{4} \varphi^{2}\wedge\bar\varphi^{2}$ and the Chern-scalar curvature is $S^{Ch} = -\frac{r^{2}}{2 \, {\left(r^{2} s^{2} - |u|^2 \right)}}<0$.
We also report the third-Chern-scalar curvature: $S^{(3)} = -\frac{r^2 \, (8 \, r^{2} s^{2} + |u|^2) }{8 \, {\left(r^{2} s^{2} - |u|^2 \right)^2}}$.
Looking at the second-Chern-Einstein tensor $\mathrm{Ric}^{(2)}(\omega)-\lambda \omega$ (whose vanishing forces $\lambda=\frac{1}{2}S^{Ch}<0$) as for the coefficient in $\varphi^1\wedge\bar\varphi^1$, we get $8\lambda r^2(r^2s^2-|u|^2)^2-r^4(4r^2s^2+5|u|^2)<0$, so the second-Chern-Einstein equation does not have any solution.

\item[Inoue $S^{\pm}$]
The structure equations can be given as $d\varphi^1=\frac{1}{2\sqrt{-1}}\varphi^{1}\wedge\varphi^{2}+\frac{1}{2\sqrt{-1}}\varphi^{2}\wedge\bar\varphi^{1}+\frac{\sqrt{-1}}{2}\varphi^{2}\wedge\bar\varphi^{2}$, $d\varphi^2=\frac{1}{2\sqrt{-1}}\varphi^{2}\wedge\bar\varphi^{2}$.
The first-Chern-Ricci form is $\mathrm{Ric}^{(1)} = - \frac{\sqrt{-1}}{2} \, \varphi^{2}\wedge\bar\varphi^{2}$ and the Chern-scalar curvature is $S^{Ch} = -\frac{r^{2}}{r^{2} s^{2} - |u|^2 }<0$.
In fact, recall that, by \cite[Remark 4.2]{teleman-mathann}, it holds $\Gamma_X(\{\omega\})<0$ for any conformal class $\{\omega\}$.
Looking at the second-Chern-Einstein tensor $\mathrm{Ric}^{(2)}(\omega)-\lambda \omega$ (whose vanishing forces $\lambda=\frac{1}{2}S^{Ch}<0$) as for the coefficient in $\varphi^1\wedge\bar\varphi^1$, we get $2\lambda r^2(r^2s^2-|u|^2)^2-r^4(r^4+r^2s^2+|u|^2+2\Re u^2)<0$ since $r^2s^2+|u|^2+2\Re u^2\geq r^2s^2+|u|^2-2|u|^2=r^2s^2-|u|^2>0$, so, there is no second-Chern-Einstein metric on $S^+$. In fact, $S^+$ admits non-trivial holomorphic vector fields by \cite[Proposition 3]{inoue}.

\item[primary Kodaira]
The structure equations can be given as $d\varphi^1=0$, $d\varphi^2=\frac{\sqrt{-1}}{2}\varphi^1\wedge\bar\varphi^1$.
It is known that the first-Chern-Ricci form is zero: $\mathrm{Ric}^{(1)} = 0$, for any invariant metric.
The Chern-scalar curvature is then zero, too; the third-Chern-scalar curvature is $ S^{(3)} = -\frac{s^{6}}{2 \, {\left(r^{2} s^{2} - |u|^2 \right)^2}} $.
Looking at the second-Chern-Einstein equation $\mathrm{Ric}^{(2)}(\omega)-\lambda \omega=0$ (whence $\lambda=\frac{1}{2}S^{Ch}=0$) as for the coefficient in $\varphi^1\wedge\bar\varphi^2$, we get $u \left( 2\lambda (r^2s^2-|u|^2)^2 - s^6 \right) = 0$, whence $u=0$; then the coefficient in $\varphi^1\wedge\bar\varphi^1$, that is $\sqrt{-1}s^4(r^2s^2-2|u|^2)$, becomes $\sqrt{-1}r^2s^6$, which is never zero.

\item[secondary Kodaira]
The structure equations can be given as $d\varphi^1=-\frac{1}{2}\varphi^1\wedge\varphi^2+\frac{1}{2}\varphi^1\wedge\bar\varphi^2$, $d\varphi^2=\frac{\sqrt{-1}}{2}\varphi^1\wedge\bar\varphi^1$.
Clearly we known that the invariant metrics are first-Chern-Ricci-flat. In particular, $S^{Ch}=0$.
Looking at the second-Chern-Einstein equation $\mathrm{Ric}^{(2)}(\omega)-\lambda \omega=0$, where in fact $\lambda=0$, as for the coefficient in $\varphi^1\wedge\bar\varphi^2$, we get $-u s^2 \left( \sqrt{-1}(r^2s^2-|u|^2) + (r^4+s^4) \right) = 0$, whence $u=0$; then the coefficient in $\varphi^1\wedge\bar\varphi^1$ becomes $\sqrt{-1}\frac{s^2}{4r^2}$, which is never zero.\qedhere
\end{description}
\end{proof}

\subsubsection{Snow manifold $S5$}\label{subsubsec:snow-zero}
Here we give an example of  a non-compact strong-(2)-Chern-Einstein manifold with constant zero Einstein factor. It is given by an invariant Hermitian metric on a four-dimensional simply-connected solvable real Lie group endowed with an invariant complex structure. Complex structures on four-dimensional Lie algebras are classified by \cite{snow-crelle, snow, ovando-manuscr, achk}, see \cite{ovando}.

More precisely, we consider the Lie group $S5$ with structure equations $[X,Y]=Y$, $[X,W]=\ell W$ with $\ell\neq0$ a real parameter, the other brackets being zero, with respect to a basis $\{X,Y,W,Z\}$. The invariant complex structure is given by the coframe of $(1,0)$-forms $\{\varphi^1,\varphi^2\}$ such that
$$ d\varphi^1=0, \quad d\varphi^2=\frac{\ell}{2}\varphi^1\wedge\varphi^2-\frac{\ell}{2}\varphi^2\wedge\bar\varphi^1.$$
It is not holomorphically-parallelizable.
Consider the generic Hermitian metric
$$ \omega = \sqrt{-1}r^2 \varphi^1\wedge\bar\varphi^1+\sqrt{-1}s^2 \varphi^2\wedge\bar\varphi^2 + u \varphi^1\wedge\bar\varphi^2-\bar u \varphi^2\wedge\bar\varphi^1 $$
where $r\neq0$, $s\neq0$, $r^2s^2-|u|^2>0$.
One sees that is it not K\"ahler.
We notice that such a metric is locally conformally K\"ahler if and only if $u=0$, and in this case the Lee form is $\theta=\frac{1}{2} \, \ell  \varphi^1 + \frac{1}{2} \, \ell  \bar\varphi^1$ (we recall that this means that $d\omega=\theta\wedge\omega$ with $d\theta=0$).
Explicit computations show that
$$ \mathrm{Ric}^{(1)}=0 , $$
and that the non-zero component of $\mathrm{Ric}^{(2)}$ are
$$
\mathrm{Ric}^{(2)}_{1\bar1} = \frac{\ell^{2} r^{2} s^{2} |u|^2}{4 \, {\left(r^{2} s^{2} - |u|^2\right)^2}} , \quad
\mathrm{Ric}^{(2)}_{1\bar2} = -\frac{\sqrt{-1} \, \ell^{2} r^{2} s^{4} u}{4 \, {\left(r^{2} s^{2} - |u|^2\right)^2}} , \quad
\mathrm{Ric}^{(2)}_{2\bar2} = \frac{\ell^{2} s^{4} |u|^2}{4 \, {\left(r^{2} s^{2} - |u|^2\right)^2}} ,
$$
and the symmetric ones.
Moreover, the Chern-scalar curvatures is clearly $S^{Ch}=0$.
Summarizing, {\em $\omega$ is always strong-(1)-Chern-Einstein with zero Einstein factor, for any value of the parameters $r$, $s$, $u$. Moreover, metrics with $u=0$ are also strong-(2)-Chern-Einstein with zero Einstein factor (in fact, they are Chern-flat)}.

For the sake of completeness, we also list the non-zero component of the third Chern-Ricci tensor: $\mathrm{Ric}^{(3)}_{1\bar2} = -\frac{\sqrt{-1} \, \ell^{2} s^{2} u}{4 \, {\left(r^{2} s^{2} - |u|^2 \right)}}$, and the third Chern-scalar curvature $S^{(3)}=-\frac{\ell^{2} s^{2} |u|^2}{2 \, {\left(r^{2} s^{2} - |u|^2\right)^2}}$.

We finally describe examples of complete negative Chern-Einstein metrics on Ovando manifolds:

\subsubsection{Ovando K\"ahler manifold $\mathfrak r_2\mathfrak r_2$}\label{subsubsec:ovando-negative-kahler}
We give an example of {\em a non-compact complete K\"ahler-Einstein manifold with negative Einstein factor}. It is given on the four-dimensional Lie group $\mathfrak{r}_2\mathfrak{r}_2=(0,-12,0,-34)$ in Salamon notation, with invariant complex structure characterized by the coframe $(\varphi^1,\varphi^2)$ of invariant $(1,0)$-forms with structure equations
$$ d\varphi^1=-\frac{1}{2}\varphi^1\wedge\bar\varphi^1 , \qquad d\varphi^2=-\frac{1}{2}\varphi^2\wedge\bar\varphi^2 . $$
We consider an invariant metric $\omega$ of the form \eqref{eq:metric}. We compute the Chern-Ricci forms:
\begin{eqnarray*}
\mathrm{Ric}^{(1)}(\omega) &=& - \frac{\sqrt{-1}}{2} \varphi^{1}\wedge\bar\varphi^1 - \frac{\sqrt{-1}}{2} \varphi^{2}\wedge\bar\varphi^2 , \\
\mathrm{Ric}^{(2)}(\omega) &=&
- \frac{\sqrt{-1} \, {\left(2 \, r^{4} s^{4} - \sqrt{-1} \, r^{2} |u|^2 \bar u - {\left(-\sqrt{-1} \, r^{2} u^{2} + {\left(r^{4} + 3 \, r^{2} s^{2}\right)} u\right)} \overline{u}\right)}}{4 \, {\left(r^{2} s^{2} - |u|^{4} \right)^2}}  \varphi^1\wedge\bar\varphi^1 \\
&& - \frac{\sqrt{-1} \, {\left(r^{2} s^{2} u^{2} - {\left(r^{2} s^{2} u - {\left(\sqrt{-1} \, r^{2} + \sqrt{-1} \, s^{2}\right)} u^{2}\right)} \overline{u}\right)}}{4 \, {\left(r^{2} s^{2} - |u|^{4} \right)^2}}  \varphi^1\wedge\bar\varphi^2 \\
&& + \frac{\sqrt{-1} \, {\left(r^{2} s^{2} |u|^2 - {\left(r^{2} s^{2} - {\left(\sqrt{-1} \, r^{2} + \sqrt{-1} \, s^{2}\right)} u\right)} \overline{u}^{2}\right)}}{4 \, {\left(r^{2} s^{2} - |u|^{4} \right)^2}} \varphi^2\wedge\bar\varphi^1 \\
&& - \left( \frac{\sqrt{-1} \, {\left(2 \, r^{4} s^{4} - \sqrt{-1} \, s^{2} |u|^2 \bar{u} - {\left(-\sqrt{-1} \, s^{2} u^{2} + {\left(3 \, r^{2} s^{2} + s^{4}\right)} u\right)} \overline{u}\right)}}{4 \, {\left(r^{2} s^{2} - |u|^{4} \right)^2}} \right)  \varphi^2\wedge\bar\varphi^2 .
\end{eqnarray*}
For the {\em K\"ahler} diagonal metric
$$ \omega_{\text{diag}} = \frac{\sqrt{-1}}{2}  \varphi^1\wedge\bar\varphi^1 + \frac{\sqrt{-1}}{2} \varphi^2\wedge\bar\varphi^2 $$
corresponding to the parameters $r=1$, $s=1$, $u=0$, we get
$$ \mathrm{Ric}^{(2)} (\omega_{\text{diag}}) = \mathrm{Ric}^{(1)} (\omega_{\text{diag}}) = - \frac{\sqrt{-1}}{2}  \varphi^1\wedge\bar\varphi^1 - \frac{\sqrt{-1}}{2} \varphi^2\wedge\bar\varphi^2 = -\omega_{\text{diag}} , $$
that is, the diagonal metric is complete K\"ahler-Einstein with negative Einstein factor.

\subsubsection{Ovando non-K\"ahler manifold $\mathfrak r_{4,-1,-1}$}\label{subsubsec:ovando-negative}
We give an example of {\em a non-compact non-K\"ahler strong-(2)-Chern-Einstein manifold with negative Einstein factor}. It is given on the four-dimensional Lie group $\mathfrak{r}_{4,\alpha,\beta}=(14,\alpha 24,\beta 34,0)$ with $\alpha=\beta=-1$, in Salamon notation, with invariant complex structure characterized by the coframe $(\varphi^1,\varphi^2)$ of invariant $(1,0)$-forms with structure equations
$$ d\varphi^1=-\frac{\sqrt{-1}}{2}\varphi^1\wedge\bar\varphi^1 , \qquad d\varphi^2=-\frac{\sqrt{-1}}{2}\varphi^1\wedge\varphi^2 -\frac{\sqrt{-1}}{2}\varphi^2\wedge\bar\varphi^1. $$
We consider an invariant metric $\omega$ of the form \eqref{eq:metric}.
By computing $d\omega=-\frac{\sqrt{-1}}{2} \bar u \varphi^{12\bar1}- \frac{1}{2} s^2 \varphi^{12\bar2} + \frac{\sqrt{-1}}{2} u \varphi^{1\bar1\bar2} - \frac{1}{2} s^2 \varphi^{2\bar1\bar2}$, we notice that $\omega$ is never K\"ahler.

We compute the Chern-Ricci forms:
\begin{eqnarray*}
\mathrm{Ric}^{(1)}(\omega) &=& - \sqrt{-1} \varphi^1 \wedge \bar\varphi^1 , \\
\mathrm{Ric}^{(2)}(\omega) &=& -\frac{\sqrt{-1} \, r^{2} s^{2}}{2 \, {\left(r^{2} s^{2} - |u|^2 \right)}} \varphi^1\wedge\bar\varphi^1 - \frac{s^{2} u}{2 \, {\left(r^{2} s^{2} - |u|^2 \right)}} \varphi^1\wedge\bar\varphi^2 \\
&& + \frac{s^{2} \overline{u}}{2 \, {\left(r^{2} s^{2} - |u|^2 \right)}}  \varphi^2\wedge\bar\varphi^1 - \frac{\sqrt{-1} \, s^{4}}{2 \, {\left(r^{2} s^{2} - |u|^2 \right)}} \varphi^2\wedge\bar\varphi^2 \\
&=& -\frac{s^{2}}{r^{2} s^{2} - |u|^2 } \cdot \left( \frac{\sqrt{-1}}{2} \, r^{2} \varphi^1\wedge\bar\varphi^1 + \frac{1}{2}u \varphi^1\wedge\bar\varphi^2 - \frac{1}{2}\overline{u} \varphi^2\wedge\bar\varphi^1 + \frac{\sqrt{-1}}{2} \, s^{2} \varphi^2\wedge\bar\varphi^2 \right) \\
&=& \frac{1}{2} S^{Ch} \cdot \omega.
\end{eqnarray*}
where the Chern-scalar curvature is
$$ S^{Ch}(\omega) = -\frac{2 \, s^{2}}{r^{2} s^{2} - |u|^2 } < 0 . $$

Clearly, $\omega$ is never (1)-Chern-Einstein, and is always strong-(2)-Chern-Einstein with negative Einstein factor.

\begin{rmk}
Also in view of \cite{gauduchon-ivanov} in complex dimension two, it would be very interesting to find (if any!) examples of negative \emph{compact} non-K\"ahler (2)-Chern-Einstein metrics.
\end{rmk}

%
%



\end{document}